\documentclass[11pt]{amsart}
\usepackage{graphicx,amssymb,amsmath}
\usepackage{a4wide}

\def\baselinestretch{1.2}

\theoremstyle{plain}
\newtheorem{theorem}{Theorem}[section]
\newtheorem{lemma}[theorem]{Lemma}
\newtheorem{corollary}[theorem]{Corollary}
\newtheorem{proposition}[theorem]{Proposition}

\theoremstyle{definition}
\newtheorem{definition}[theorem]{Definition}
\newtheorem{example}[theorem]{Example}
\newtheorem{remark}[theorem]{Remark}

\def\case#1#2{\par\addvspace{\medskipamount}%
    \noindent \underline{\emph{Case #1. #2\,}}\par\smallskip}

\newcommand{\thickhline}{\noalign{\hrule height 1pt}}

\let\epsilon\varepsilon

\def\Z{\mathbb Z}

\def\syl{\mathrm{syl}}
\def\esyl{\mathrm{esyl}}
\def\len{\mathrm{len}}
\def\infs{{\inf{\!}_{s}}}

\let\prec\preccurlyeq

\begin{document}

\title{Unknotting number and genus of 3-braid knots}

\author{Eon-Kyung Lee and Sang-Jin Lee}
\address{Department of Mathematics, Sejong University,
    Seoul, 143-747, Korea}
\email{eonkyung@sejong.ac.kr}
\address{Department of Mathematics, Konkuk University,
    Seoul, 143-701, Korea}
\email{sangjin@konkuk.ac.kr}
\date{\today}

\begin{abstract}

Let $u(K)$ and $g(K)$ denote the unknotting number and the genus of
a knot $K$, respectively.
For a 3-braid knot $K$, we show that $u(K)\le g(K)$ holds,
and that if $u(K)=g(K)$ then $K$ is
either a 2-braid knot, a connected sum of two 2-braid knots, the
figure-eight knot, a strongly quasipositive knot or its mirror image.

\end{abstract}

\maketitle

\tableofcontents

\section{Introduction}\label{sec:Intro}

A \emph{knot} is a connected closed 1-manifold, smoothly embedded in the 3-sphere $S^3$.
The \emph{unknotting number} $u(K)$ of a knot $K$ is the minimum number of crossing changes
needed to transform $K$ into the unknot.
The \emph{genus} $g(K)$ (resp.\ the \emph{4-genus} $g^*(K)$)
of a knot $K$ in $S^3=\partial B^4$
is the minimum genus over all orientable surfaces in $S^3$
(resp.\ in $B^4$) with boundary $K$.
For a knot $K$, both $u(K)$ and $g(K)$ are bounded
from below by $g^*(K)$:
$$
g^*(K)\le u(K)\quad\mbox{and}\quad g^*(K)\le g(K).
$$
However, there is no such inequality between $u(K)$ and $g(K)$ in general.
There are knots whose genus is greater than its unknotting number,
and vice versa.

In this paper we study the relationship between the unknotting number and
the genus of 3-braid knots. We first establish the following theorem.

\begin{theorem}\label{thm:main}
If $K$ is a 3-braid knot,
then
$$g^*(K)\le u(K)\le g(K).$$
\end{theorem}

We remark that the inequality $u(K)\le g(K)$ does not
hold for knots with braid index $\ge 4$.
There is a knot with unknotting number 2 and genus 1,
given by Livingston~\cite[Appendix]{ST88}.
According to the database \textsc{KnotInfo}
of Cha and Livingston~\cite{CL},
there are 43 knots with braid index 4 and with crossing number $\le 9$,
and among them there are 6 knots with $u(K)>g(K)$
as listed in Table~\ref{tab:knotinfo}.

\begin{table}[t]
\let\s\sigma
$$
\begin{array}{*5c}\thickhline
\mbox{$K$} & \mbox{$g^*(K)$} & \mbox{$u(K)$} & \mbox{$g(K)$}
& \mbox{Braid representatives}\\\thickhline
7_4    &1 & 2 & 1 & \s_1\s_1\s_2\s_1^{-1}\s_2\s_2\s_3\s_2^{-1}\s_3\\\hline
9_{10} &2 & 3 & 2 & \s_1\s_1\s_2\s_1^{-1}\s_2\s_2\s_2\s_2\s_3\s_2^{-1}\s_3\\\hline
9_{13} &2 & 3 & 2 & \s_1\s_1\s_1\s_1\s_2\s_1^{-1}\s_2\s_2\s_3\s_2^{-1}\s_3\\\hline
9_{38} &2 & 3 & 2 & \s_1\s_1\s_2\s_2\s_3^{-1}\s_2\s_1^{-1}\s_2\s_3\s_3\s_2\\\hline
9_{46} &0 & 2 & 1 & \s_1\s_2^{-1}\s_1\s_2^{-1}\s_3\s_2\s_1^{-1}\s_2\s_3\\\hline
9_{49} &2 & 3 & 2 & \s_1\s_1\s_2\s_1\s_1\s_3^{-1}\s_2\s_1^{-1}\s_2\s_3\s_3\\\thickhline
\end{array}
$$
\caption{4-braid knots $K$ with $u(K)>g(K)$ and crossing number $\le 9$.}
\label{tab:knotinfo}
\end{table}

\medskip

A braid is said to be \emph{strongly quasipositive} if it can be written as
a positive word in band generators. (See \S2 for the definition of band generators.)
A knot which is the closure of such a braid is called
a \emph{strongly quasipositive knot}.
It is known by Rudolph~\cite[Corollary]{Rud93} that $g^*(K) = g(K)$
holds for strongly quasipositive knots $K$ with arbitrary braid index.
Therefore we have the following corollary.

\begin{corollary}\label{cor:SQP}
If $K$ is a strongly quasipositive 3-braid knot, then
$$g^*(K)=u(K)=g(K).$$
\end{corollary}

We remark that the equality $g^*(K)=u(K)=g(K)$ holds
for braid positive knots with arbitrary braid index,
but \emph{not} for strongly quasipositive knots and
\emph{not} for positive knots with braid index $\ge 4$.
Recall that a knot is \emph{positive} if it has a diagram
consisting of only positive crossings
and that a knot is \emph{braid positive} if it is a closure of a braid which can
be written as a positive word in Artin generators.
Thus a braid positive knot is a positive knot.
It was shown independently by Rudolph~\cite{Rud99} and Nakamura~\cite{Nak00}
that positive knots are strongly quasipositive.
Hence we have the following implications.
$$
\mbox{braid positive} \quad\Longrightarrow\quad
\mbox{positive} \quad\Longrightarrow\quad
\mbox{strongly quasipositive}
$$

Because $g^*(K) = g(K)$ holds for strongly quasipositive knots~\cite{Rud93},
it also holds for positive knots and braid positive knots.
Rasmussen also proved the equality $g^*(K)=g(K)$ for positive knots
by using his invariant~\cite{Ras10}.

Stoimenow~\cite{Sto03} proved that
$u(K)= g(K)$ holds for braid positive knots $K$,
using an inequality of Boileau-Weber-Rudolph~\cite{BW84,Rud83}.
Therefore $g^*(K)=u(K)=g(K)$ holds for braid positive knots
with arbitrary braid index.

For knots in Table~\ref{tab:knotinfo} except $9_{46}$,
we have $g^*(K)=g(K)<u(K)$.
By a straightforward computation,
we can see that these knots are strongly quasipositive 4-braid knots.
Moreover, they are positive knots.
Therefore the equality in Corollary~\ref{cor:SQP}
does not hold for positive knots
(and hence not for strongly quasipositive knots)
with braid index $\ge 4$.

\medskip

Our last theorem shows that for 3-braid knots
the equality $u(K) = g(K)$ holds only for special cases.
Therefore the strict inequality $u(K)< g(K)$ holds for generic 3-braid knots.

\begin{theorem}\label{thm:equal}
Let $K$ be a 3-braid knot.
If $u(K)=g(K)$, then one of the following holds:
\begin{enumerate}
\item Either $K$ or its mirror is strongly quasipositive,
hence $K$ is represented either by a positive word or by a negative word in band generators;
\item $K$ is the figure-eight knot,
hence $K$ is represented by $\sigma_1^{-1}\sigma_2\sigma_1^{-1}\sigma_2$;
\item $K$ is a 2-braid knot,
hence $K$ is represented by $\sigma_1^p\sigma_2^\epsilon$
where $p$ is an odd integer and $\epsilon=\pm 1$;
\item $K$ is a connected sum of two 2-braid knots,
hence $K$ is represented by $\sigma_1^p\sigma_2^q$
where $p$ and $q$ are odd integers.
\end{enumerate}
\end{theorem}

We remark that the converse of the above theorem is an open question.
It is obvious that $u(K)=g(K)$ holds for the knots of types (i), (ii) and (iii).
(In fact, the equality $g^*(K)=u(K)=g(K)$ holds for these knots.)
However, it is unknown whether the equality holds for the knots of type (iv).
This leads to the following question.

\medskip\noindent
\textbf{Question.}\ \
Is it true that if $K$ is the closure of the 3-braid $\sigma_1^{2p+1}\sigma_2^{-2q-1}$
with $p,q\ge1$ then $u(K)=p+q$?

\medskip

This is a special case of the long-standing conjecture that
the unknotting number is additive under connected sum, which is still open.
If the above question has an affirmative answer,
then the converse of Theorem~\ref{thm:equal} is also true,
which gives a complete classification of 3-braid knots
with unknotting number equal to genus.

In the above question, if $p=q$, then the knot $K$ is a ribbon knot, hence $g^*(K)=0$.
Therefore, in this case, the 4-genus $g^*(K)$ does not give any information
about the unknotting number of $K$.
Up to 10 crossings, there are four knots exclusively of type (iv), namely,
$3_1\sharp !3_1$,
$3_1\sharp\, !5_1$,
$3_1\sharp\, !7_1$ and
$5_1\sharp\, !5_1$,
which are represented by the braids
$\sigma_1^3\sigma_2^{-3}$,
$\sigma_1^3\sigma_2^{-5}$,
$\sigma_1^3\sigma_2^{-7}$ and
$\sigma_1^5\sigma_2^{-5}$, respectively.
Here $!K$ denotes the mirror image of $K$ with reverse orientation.
The unknotting number of $3_1\sharp\, !3_1$ is 2:
the unknotting number is at most 2 because
the knot can be transformed into the unknot by 2 crossing changes,
and it is at least 2 because the unknotting number of a composite knot is at least 2
by a result of Scharlemann~\cite{Sch85}.
However, according to the table of unknotting numbers of composite knots
up to 10 crossings given by Stoimenow in~\cite[Appendix I]{Sto04},
the unknotting numbers of the knots
$3_1\sharp\, !5_1$,
$3_1\sharp\, !7_1$ and
$5_1\sharp\, !5_1$ are unknown.

The above question was also asked by Abe, Hanaki and Higa in~\cite{AHH12}.
In that paper, they showed that if $K$ is a knot with $u(K)=(c(K)-2)/2$,
then $K$ is either the figure-eight knot, a positive 3-braid knot,
a negative 3-braid knot or a connected sum of two 2-braid knots.
Here $c(K)$ denotes the crossing number of $K$
and a positive (resp.\ negative) 3-braid knot means the closure of a 3-braid which is
represented by a positive (resp.\ negative) word in Artin generators.
An affirmative answer to the above question gives
a complete classification of the knots with $u(K)=(c(K)-2)/2$.

\medskip
We close this section by explaining briefly the ideas of the proofs of
Theorems~\ref{thm:main} and~\ref{thm:equal}.

The proof of Theorem~\ref{thm:main} is based on the following:
any nontrivial 3-braid knot is represented by a word
$Wa_1^{\pm 2}$ (in band generators) that is a shortest word
in its conjugacy class.
Similar arguments were used by Ni \cite{Ni09} and Stoimenow \cite{Sto06}
in studying the fibredness of 3-braid knots.

Analyzing shortest words of 3-braids more carefully, we obtain the following:
if $K$ is a 3-braid knot other than those listed in Theorem~\ref{thm:equal},
then $K$ is represented by a word
$a_1a_2^{2k}a_1^{-1}W$ (in band generators) that is a shortest word in its conjugacy class.
This is the key idea of the proof of Theorem~\ref{thm:equal}.

\section{The 3-braid group $B_3$}
The braid groups $B_n$, $n\ge 2$, have the presentation
$$
B_n  =  \left\langle \sigma_1 ,\ldots, \sigma_{n-1} \left|
\begin{array}{l}
\sigma_i\sigma_j=\sigma_j\sigma_i\quad  \mbox{if $|i-j|\ge 2$}, \\
\sigma_i\sigma_j\sigma_i=\sigma_j\sigma_i\sigma_j\quad  \mbox{if $|i-j|=1$}.
\end{array}
\right.\right\rangle.
$$
This presentation is called the \emph{Artin presentation}
and the generators $\sigma_i$ are called \emph{Artin generators}.
In the late nineties,
Birman, Ko and Lee~\cite{BKL98} introduced a then-new presentation,
which we call the {\em dual presentation},
$$
B_n  =  \left\langle a_{ij}, \ 1\le j < i\le n \left|
\begin{array}{l}
a_{ij}a_{kl}=a_{kl}a_{ij} \quad \mbox{if $(k-i)(k-j)(l-i)(l-j)>0$}, \\
a_{ij}a_{jk}=a_{jk}a_{ik}=a_{ik}a_{ij} \quad \mbox{if $1\le k<j<i \le n$}.
\end{array}
\right.\right\rangle.
$$
The generators $a_{ij}$ are called {\em band generators}.
They are related to Artin generators by
$a_{ij}=\sigma_{i-1}\cdots\sigma_{j+1}\sigma_j\sigma_{j+1}^{-1}
\cdots\sigma_{i-1}^{-1}$.
See Figure~\ref{fig:gen}(a).

Let $B_n^+$ denote the monoid generated by band generators.
Elements of $B_n^+$ are called \emph{strongly quasipositive braids}~\cite{Rud93}.
In this paper, we simply call them \emph{positive braids}.
If a knot is the closure of a positive braid,
it is called a \emph{strongly quasipositive knot}.
For example, the knot $9_{49}$ in Table~\ref{tab:knotinfo} is a strongly quasipositive
4-braid knot because the knot is represented by the braid
$\let\s\sigma \s_1\s_1\s_2\s_1\s_1\s_3^{-1}\s_2\s_1^{-1}\s_2\s_3\s_3$
which is conjugate to
\let\s\sigma
\begin{align*}
(\s_3 &\s_1^{-1})\, \s_1\s_1\s_2\s_1\s_1\s_3^{-1}\s_2
\s_1^{-1}\s_2\s_3\s_3  \,(\s_3^{-1}\s_1 )\\
&=\s_3\s_1\s_2\s_1\s_1\s_3^{-1}\s_2\s_1^{-1}\s_2\s_3\s_1\\
&=\s_1(\s_3\s_2\s_3^{-1})\s_1\s_1\s_2 (\s_1^{-1}\s_2\s_1)\s_3\\
&=a_{21}a_{42}a_{21}a_{21}a_{32}a_{31}a_{43}.
\end{align*}
In the above, we use the identities $\sigma_3\sigma_1=\sigma_1\sigma_3$,
$\sigma_i=a_{(i+1)i}$,
$\sigma_3\sigma_2\sigma_3^{-1}=a_{42}$ and
$\sigma_1^{-1}\sigma_2\sigma_1=\sigma_2\sigma_1\sigma_2^{-1}=a_{31}$.
Similarly, all the knots in Table~\ref{tab:knotinfo} except $9_{46}$ are
strongly quasipositive 4-braid knots.

\begin{figure}[t]
$$
\begin{array}{*6{c}}
\includegraphics[scale=.4]{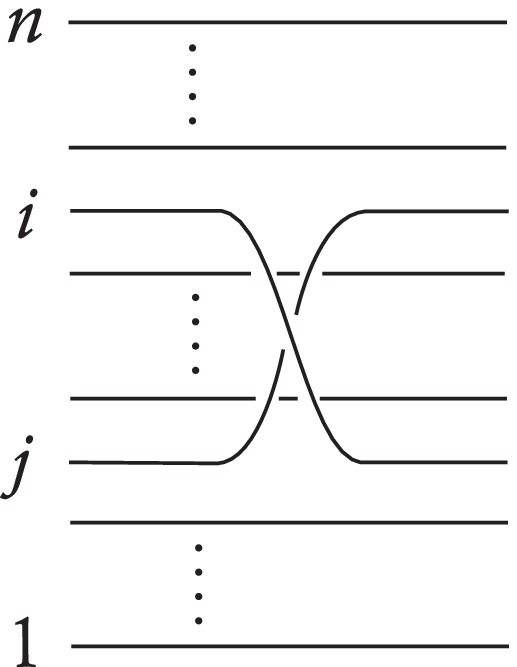}  &\quad&
\includegraphics[scale=.55]{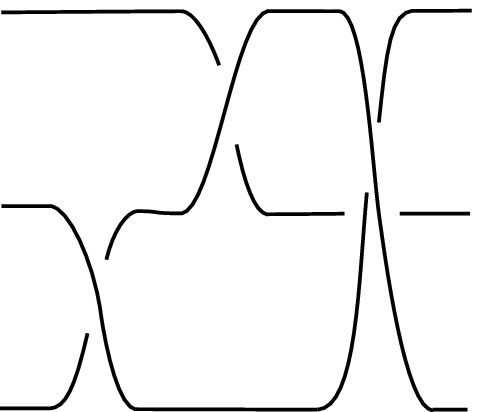} &\quad&
\includegraphics[scale=.8]{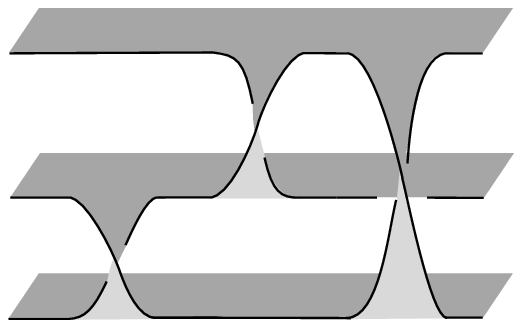} \\[2ex]
\mbox{(a) Band generator $a_{ij}$} &&
\mbox{(b) 3-braid $a_1a_2^{-1}a_3$} &&
\mbox{(c) Banded surface for $a_1a_2^{-1}a_3$}
\end{array}$$
\caption{}\label{fig:gen}
\end{figure}

For the 3-braid group $B_3$, we use the notations
$a_1 = a_{21}$, $a_2 = a_{32}$ and $a_3 = a_{31}$.
Then the dual presentation becomes
$$
B_3  =  \left\langle\, a_1, a_2, a_3 \mid
a_2 a_1 = a_3 a_2 = a_1 a_3
\,\right\rangle.
$$
This presentation was first studied by Xu~\cite{Xu92}.
The generators $a_i$ are related to Artin generators by
$a_1=\sigma_1$, $a_2=\sigma_2$ and $a_3=\sigma_2\sigma_1\sigma_2^{-1}$.
From now on, we use a convention of taking modulo 3
for the indices of the generators
$a_1, a_2, a_3$.
For example, $a_{-2} = a_1 = a_4$.

\begin{definition}[word length, syllable length
and nondecreasing positive words]
Let
$$W = a_{i_1}^{k_1} a_{i_2}^{k_2} \cdots a_{i_r}^{k_r}$$
be a word in band generators of $B_3$,
where $k_j\ne 0$ for all $j$ and $i_j\ne i_{j+1}$ for $j=1,\ldots,r-1$.
The \emph{syllable length} and \emph{word length} of $W$ are defined
respectively as
$$
\syl(W)=r\quad\mbox{and}\quad |W|=|k_1|+\cdots+|k_r|.
$$
For a 3-braid $\alpha$, let
$$|\alpha|=\min\{\,|W|:\mbox{$W$ is a word in band generators representing $\alpha$}\,\}.$$
If each $k_j \ge 1$, we call $W$ a {\em positive word}.
Notice that a positive word represents a positive braid.
If there is no confusion,
we will not distinguish between a word and the braid represented by the word.
A positive word is said to be \emph{nondecreasing} if the indices
of the generators in the word are nondecreasing, that is,
$i_{j+1}=i_j+1$ for $j=1,\ldots,r-1$.
For example, $a_1a_2^2a_3a_1^3$ is a
nondecreasing positive word of word length 7 and syllable length 4.
\end{definition}

We denote the closure of a braid $\alpha$ by $\hat\alpha$.
If $\alpha$ and $\beta$ are conjugate,
then $\hat\alpha=\hat\beta$.
If $\beta=\alpha^{-1}$, then $\hat\beta$ is the mirror image of $\hat\alpha$
with reverse orientation,
in particular, $g(\hat\alpha)=g(\hat\beta)$ and $u(\hat\alpha)=u(\hat\beta)$.

\begin{definition}[banded surface]
Let $W$ be a word in band generators representing a 3-braid $\alpha$.
The closure $\hat\alpha$ has a Seifert surface
which consists of three horizontal disks and half-twisted bands
each corresponding to a band generator or its inverse in the word $W$.
We call such a surface the {\em banded surface} for $W$, denoted by $F_W$.
See Figure~\ref{fig:gen}(b,c).
\end{definition}

The Euler characteristic of $F_W$ is $\chi(F_W)=3-|W|$.
If $W$ represents a knot, then $\partial F_W$ has one component, hence
$g(F_W)=|W|/2-1$.

\begin{theorem}[Bennequin~\cite{Ben83}]
Let $W$ be a word in band generators representing a 3-braid $\alpha$ .
If $\hat\alpha$ is a knot and
if the word length of $W$ is minimal in the conjugacy class of $\alpha$,
then $F_W$ is a minimal genus Seifert surface of $\hat\alpha$,
hence
$$
g(\hat\alpha) = g(F_W) = |W|/2-1.
$$
\end{theorem}

\begin{definition}[subword order]
For positive 3-braids $P$ and $Q$, we write $P\prec Q$
if $Q=R_1PR_2$ for some $R_1,R_2\in B_3^+$.
This gives rise to a partial order on $B_3^+$.
\end{definition}

For example, $a_1a_2^2\prec a_3a_1a_2^3=a_3(a_1a_2^2)a_2$ and $a_3\prec a_2a_1=a_3a_2$.

\begin{definition}[fundamental braid and rotation automorphism]
The \emph{fundamental braid} $\delta$ is defined as $\delta = a_2a_1$.
The \emph{rotation automorphism} $\tau:B_3\to B_3$
is defined as $\tau(a_i)=a_{i+1}$ for $i=1,2,3$.
\end{definition}

It is easy to see that $a_i\delta=\delta a_{i+1}$ for all $i$,
hence $\delta^3$ is a central element and
$\tau(\alpha)=\delta^{-1}\alpha\delta$ for $\alpha\in B_3$.

\begin{definition}[Garside normal form]
Every 3-braid $\alpha$ is uniquely expressed as
$$\alpha=\delta^u P,$$
where $u$ is an integer and $P$ is a nondecreasing positive word
in band generators~\cite{Xu92}.
The above expression is called the \emph{Garside normal form} of $\alpha$.
The \emph{infimum}, \emph{supremum} and  \emph{canonical length}
of $\alpha$ are defined respectively as
$$
\inf(\alpha)=u,\quad\sup(\alpha)=u+|P|\quad\mbox{and}\quad
\len(\alpha)=|P|.
$$
The \emph{syllable length} of $\alpha$ is defined as
the syllable length of the positive word $P$:
$$
\syl(\alpha)=\syl(P).
$$
The \emph{extended syllable length} of $\alpha$ is defined as
$$
\esyl(\alpha)=\inf(\alpha)+\syl(\alpha).
$$
\end{definition}

Infimum, supremum, canonical length and (extended) syllable length
are invariant under the automorphism $\tau$.

The Garside normal form can be obtained by performing the
following transformations repeatedly to a word representing $\alpha$:
\begin{enumerate}
\item replace $a_i^{-1}$ with $\delta^{-1}a_{i+1}$;
\item replace $a_{i+1}a_i$ with $\delta$;
\item replace $\delta^k\delta^\ell$ with $\delta^{k+\ell}$;
\item replace $a_i\delta^k$ with $\delta^{k}a_{i+k}$.
\end{enumerate}

\begin{definition}[right complement]
For a positive 3-braid $P$,
the braid $P^{-1}\delta^{|P|}$ is called
the \emph{right complement} of $P$, denoted by $P^*$.
\end{definition}

The right complement $P^*$ is the positive braid such that $PP^*=\delta^{|P|}$.

\begin{lemma}\label{lem:RC}
The following hold for 3-braids.
\begin{enumerate}
\item $(P_1P_2)^*=P_2^*\cdot\tau^{|P_2|}(P_1^*)$
    for positive braids $P_1$ and $P_2$.
\item $(a_{i_1}a_{i_2}\cdots a_{i_r})^*
    =a_{i_r+2}a_{i_{r-1}+3}\cdots a_{i_1+r+1}$ for $r\ge 1$.
\item $(a_i^r)^*=a_{i+2}a_{i+3}\cdots a_{i+r+1}$ for $r\ge 1$.
\item $(a_{i+1}a_{i+2}\cdots a_{i+r})^*=a_{i+r+2}^r$ for $r\ge 1$.
\item If $P$ is a nondecreasing positive word,
    then so is $P^*$ with $|P^*|=|P|$.
\item If $\alpha=\delta^uP$ is the Garside normal form of $\alpha$, then
$\alpha^{-1}=\delta^{-\sup(\alpha)} \tau^{-\sup(\alpha)}(P^*)$
is the Garside normal form of $\alpha^{-1}$. In particular,
$$
\inf(\alpha^{-1})=-\sup(\alpha),\
\sup(\alpha^{-1})=-\inf(\alpha)\ \mbox{and}\
\len(\alpha^{-1})=\len(\alpha).
$$
\end{enumerate}
\end{lemma}

\begin{proof}
(i)\ \ $(P_1P_2)^*=P_2^*\cdot\tau^{|P_2|}(P_1^*)$ because
$$
P_1P_2\cdot P_2^*\cdot\tau^{|P_2|}(P_1^*)
=P_1\cdot\delta^{|P_2|} \cdot \tau^{|P_2|}(P_1^*)
=P_1 \cdot P_1^*\cdot \delta^{|P_2|}
= \delta^{|P_1|}\delta^{|P_2|}
= \delta^{|P_1P_2|}.
$$

(ii)-(v)\ \
Observe that $a_i^*=a_{i+2}$ because $a_ia_{i+2}=\delta$.
Then (ii) follows from (i),
and (iii)-(v) follow from (ii).

\smallskip
(vi)\ \
Notice that
$$
\alpha^{-1}=P^{-1}\delta^{-u}=P^{-1}\delta^{|P|}\delta^{-u-|P|}
=P^*\delta^{-\sup(\alpha)}
=\delta^{-\sup(\alpha)} \tau^{-\sup(\alpha)}(P^*).
$$
Because $P$ is nondecreasing, so is $P^*$ by (v).
Therefore
$\alpha^{-1}=\delta^{-\sup(\alpha)} \tau^{-\sup(\alpha)}(P^*)$
is the Garside normal form of $\alpha^{-1}$.
The identities are immediate from this and (v).
\end{proof}

For example,
$(a_1^3a_2a_3a_1^7)^*
= (a_1^2\cdot \underbrace{a_1a_2a_3a_1}_4\cdot a_1^6)^*
= \underbrace{a_3a_1a_2a_3a_1a_2}_6 \cdot a_3^4\cdot  \underbrace{a_1a_2}_2$.

\begin{definition}[summit set]
For a 3-braid $\alpha$, we define
$$
\infs(\alpha)=\max\{\,\inf(\beta)\mid\mbox{$\beta$ is conjugate to $\alpha$}\}.
$$
The \emph{summit set} of $\alpha$ is defined as
$$
[\alpha]^S=\{\,\beta\in B_3 \mid
\mbox{$\beta$ is conjugate to $\alpha$ and
$\inf(\beta)=\infs(\alpha)$}\,\}.
$$
Elements of $[\alpha]^S$ are called \emph{summit elements}.
We define a subset $[\alpha]^S_0$ of $[\alpha]^S$ as
$$
[\alpha]^S_0=\left\{\,\beta\in[\alpha]^S \mid
\mbox{$\syl(\beta)$ is minimal in the conjugacy class of $\alpha$}
\,\right\}.
$$
\end{definition}

It is known that for a 3-braid $\alpha$, the sets $[\alpha]^S$ and $[\alpha]_0^S$
are finite nonempty subsets of the conjugacy
class of $\alpha$ and can be computed in a finite number of steps.
The following are equivalent for a 3-braid $\alpha$:
(i) $\inf(\alpha)$ is maximal in the conjugacy class;
(ii) $\sup(\alpha)$ is minimal in the conjugacy class;
(iii) $\len(\alpha)$ is minimal in the conjugacy class.

\begin{definition}[positive conjugate]
Let $\alpha=\delta^u P$ be the Garside normal form of $\alpha\in B_3$.
Let $P=P_1P_2$ for positive words $P_1$ and $P_2$.
Then a 3-braid $\beta$ is called a \emph{positive conjugate} of $\alpha$
if $\beta$ is either $\delta^u\tau^u(P_2)P_1$ or $\delta^u P_2\tau^{-u}(P_1)$.
\end{definition}

In the above definition, $\beta$ is a conjugate of $\alpha$ because, for example,
$\delta^u\tau^u(P_2)P_1=P_2\alpha P_2^{-1}$;
if $\alpha$ is a summit element, then so is $\beta$ because
$\inf(\beta)\ge \inf(\alpha)$ and $\inf(\alpha)$ is maximal in the conjugacy class;
if $\alpha\in[\alpha]^S_0$ and $\syl(P_1P_2)=\syl(P_1)+\syl(P_2)$
(i.e.\ the last letter of $P_1$ is different from the first letter of
$P_2$), then $\beta\in[\alpha]^S_0$.

\smallskip

Let $\alpha=\delta^u a_1^{k_1}a_2^{k_2}\cdots a_r^{k_r}$,
where $k_1,\ldots,k_r\ge 1$.
Observe the following.

\begin{itemize}
\item
If $\len(\alpha)\ge 2$ and $\esyl(\alpha)=u+r\equiv 2\bmod 3$,
then $\alpha\not\in[\alpha]^S$
because $\inf(\beta)=\inf(\alpha)+1$
for a positive conjugate
\begin{align*}
\beta
&=\delta^u \tau^u(a_r) a_{1}^{k_1}a_{2}^{k_2}\cdots a_{r-1}^{k_{r-1}}a_{r}^{k_r-1}\\
&=\delta^u a_{2}a_{1}a_{1}^{k_1-1}a_{2}^{k_2}\cdots a_{r-1}^{k_{r-1}}a_{r}^{k_r-1}\\
&=\delta^{u+1} a_{1}^{k_1-1}a_{2}^{k_2}\cdots a_{r-1}^{k_{r-1}}a_{r}^{k_r-1}.
\end{align*}

\item
If $\syl(\alpha)\ge 2$ and $\esyl(\alpha)=u+r\equiv 1\bmod 3$,
then $\alpha\not\in[\alpha]^S_0$ because
$\syl(\beta)=\syl(\alpha)-1$
for a positive conjugate
\begin{align*}
\beta
&=\delta^u \tau^u(a_{r}^{k_r}) a_{1}^{k_1}a_{2}^{k_2}\cdots a_{r-1}^{k_{r-1}}\\
&=\delta^u a_{1}^{k_r}a_{1}^{k_1}a_{2}^{k_2}\cdots a_{r-1}^{k_{r-1}}\\
&=\delta^{u} a_{1}^{k_1+k_r}a_{2}^{k_2}\cdots a_{r-1}^{k_{r-1}}.
\end{align*}
Moreover, $\beta\in[\alpha]^S_0$ by the following lemma because
$\esyl(\beta)
= u+(r-1) \equiv 0\bmod 3$.
\end{itemize}

\begin{lemma}[\cite{Xu92,KL99}]\label{lem:xu}
Let $\alpha$ be a 3-braid.
\begin{enumerate}
\item
$\alpha\in[\alpha]^S$ if and only if $\alpha^{-1}\in[\alpha^{-1}]^S$.
\item
$\alpha\in[\alpha]^S$ if and only if\/ $\len(\alpha)\le 1$
or $\esyl(\alpha)\not\equiv 2\bmod 3$.
\item
Let $\alpha\in[\alpha]^S$. Then
$\alpha\in[\alpha]^S_0$ if and only if\/
$\syl(\alpha)\le 1$
or $\esyl(\alpha)\equiv 0\bmod 3$.
\end{enumerate}
\end{lemma}

\begin{lemma}[shortest word, \cite{Xu92}]\label{lem:short}
Let $\alpha=\delta^u P$ be the Garside normal form of $\alpha\in B_3$.
\begin{enumerate}
\item
If $u\ge 0$, then $(a_2a_1)^u P$ is a shortest word for $\alpha$.

\item
If $-|P|<u<0$, then
$(P_1^*)^{-1}P_2$ is a shortest word for $\alpha$,
where $P_1$ and $P_2$ are positive words such that $P=P_1P_2$ and $|P_1|=-u$.
In particular, $|\alpha|=|P|$.

\item
If $u\le -|P|$, then $(a_2a_1)^{-\ell}(P^*)^{-1}$ is a shortest word for $\alpha$,
where $\ell=-u-|P|$.

\item If $\alpha$ is a summit element, then $\alpha$ has the shortest
word length in its conjugacy class, that is, $|\alpha|\le|\beta|$
whenever $\beta$ is conjugate to $\alpha$.

\item If $\alpha$ is a summit element and $\hat\alpha$ is a knot,
then $g(\hat\alpha)=|\alpha|/2-1$.
\end{enumerate}
\end{lemma}

\section{Proof of Theorem~\ref{thm:main}}
\label{sec:thm1}

\begin{lemma}\label{lem:A1}
Let $\alpha$ be a 3-braid such that $\hat\alpha$ is a knot and $\syl(\alpha)=1$.
Then $\alpha$ is of the form $\alpha=\delta^u a_i^{2p}$
for some $p\ge 1$.
Moreover, if $\alpha\in[\alpha]^S$, then
$\esyl(\alpha)\equiv 0\bmod 3$.
\end{lemma}

\begin{proof}
Applying $\tau$ if necessary, we may assume that $\alpha=\delta^u a_1^k$ for some $k\ge 1$.
If $k\equiv 1\bmod 2$, then the induced permutation of $\alpha$ is the same as
that of $a_1$, $\delta a_1$ or $\delta^{-1}a_1$,
hence it has two cycles.
This contradicts that $\hat\alpha$ is a knot.
Therefore $k=2p$ for some $p\ge 1$, hence
$$\alpha=\delta^u a_1^{2p}.$$

Suppose $\alpha\in[\alpha]^S$.
Since $\len(\alpha)=2p\ge 2$, $\esyl(\alpha)\equiv 0,1\bmod 3$ by Lemma~\ref{lem:xu}.
If $\esyl(\alpha)\equiv 1\bmod 3$, then
$u=\inf(\alpha)=\esyl(\alpha)-\syl(\alpha)\equiv 1-1\equiv 0\bmod 3$.
Thus the induced permutation of $\alpha$ is the identity,
which contradicts that $\hat\alpha$ is a knot.
Therefore $\esyl(\alpha)=0\bmod 3$.
\end{proof}

\begin{lemma}\label{lem:A}
Let $\alpha$ be a 3-braid such that $\hat\alpha$ is a knot.
Suppose that $\alpha\in[\alpha]^S_0$ and $\syl(\alpha)\ne 0$.
Let $\alpha=\delta^u a_1^{k_1}a_2^{k_2}\cdots a_r^{k_r}$
be the Garside normal form.
Then
\begin{enumerate}
\item $\esyl(\alpha)\equiv 0\bmod 3$;
\item $k_i\equiv 0\bmod 2$ for some $1\le i\le r$, in particular $k_i\ge 2$.
\end{enumerate}
\end{lemma}

\begin{proof}
(i)\ \
It follows from Lemma~\ref{lem:A1} (when $\syl(\alpha)=1$)
and Lemma~\ref{lem:xu} (when $\syl(\alpha)\ge 2$).

\smallskip
(ii)\ \
Assume that $k_i\equiv 1\bmod 2$  for all $1\le i\le r$.
Since $u\equiv -r\bmod 3$ by (i),
the induced permutation of $\alpha=\delta^u a_1^{k_1}\cdots a_r^{k_r}$
is the same as that of
$$\delta^{-r}a_1a_2\cdots a_r=[(a_1a_2\cdots a_r)^*]^{-1}=a_{r+2}^{-r},
$$
hence it has two or three cycles.
This contradicts that $\hat\alpha$ is a knot.
\end{proof}

\begin{corollary}\label{cor:A1}
Let $\alpha$ be a 3-braid such that $\hat\alpha$ is a knot and $\alpha\in[\alpha]^S$.
If\/ $\syl(\alpha)\ne 0$, then $\alpha$ is conjugate to $\delta^u P a_1^2$
where $u=\inf(\alpha)$ and $Pa_1^2$ is a nondecreasing positive word.
\end{corollary}

\begin{proof}
Taking a conjugate if necessary,
we may assume $\alpha\in[\alpha]^S_0$.
Let $\alpha=\delta^u a_{j+1}^{k_1}a_{j+2}^{k_2}\cdots a_{j+r}^{k_r}$
be the Garside normal form.
By Lemma~\ref{lem:A}, $k_i\ge 2$ for some $1\le i\le r$.
Taking a positive conjugate if necessary, we may assume $k_r\ge 2$.
Applying $\tau$ if necessary, we may assume $j+r\equiv 1\bmod 3$.
Therefore $\alpha=\delta^u Pa_1^2$ for some nondecreasing positive word $Pa_1^2$.
\end{proof}

The following corollary will be used in the proof of Theorem~\ref{thm:equal}
in \S4.

\begin{corollary}\label{cor:general}
Let $\alpha$ be a 3-braid such that $\hat\alpha$ is a knot.
Suppose that the minimal syllable length
in the conjugacy class of $\alpha^{-1}$ is at least 3.
Then $\alpha$ is conjugate to
$$
\beta=\delta^u a_1^{2k} Q_1
\underbrace{a_{i+1}a_{i+2}\cdots a_{i+2p-2}}_{2p-2} Q_2
$$
with $u=\infs(\alpha)$ and $k,p\ge 1$ such that
\begin{enumerate}
\item $\beta\in[\alpha]_0^S$;
\item $Q_1$ and $Q_2$ are nondecreasing positive words,
possibly being the empty word;
\item if $Q_1$ is not the empty word,
then $Q_1$ starts with $a_2$ and ends with $a_i^2$;
\item if $Q_2$ is not the empty word,
then $Q_2$ starts with $a_{i+2p-1}^2$ and ends with $a_{3-u}$.
\end{enumerate}
In particular, if $|Q_j|\ne 0$ then $|Q_j|\ge 2$ for $j=1,2$.
\end{corollary}

\begin{proof}
We may assume that $\alpha^{-1}\in[\alpha^{-1}]^S_0$.
(In particular, $\alpha^{-1}\in[\alpha^{-1}]^S$, hence $\alpha\in[\alpha]^S$.)
Let $\alpha^{-1}=\delta^{\inf(\alpha^{-1})}Q$ be the Garside normal form.
By Lemma~\ref{lem:A}(ii), taking a positive conjugate if necessary,
we may assume that
$$
a_1a_2^{2p}a_3\prec Q
$$
for some $p\ge 1$.
By Lemma~\ref{lem:RC}, the Garside normal form of $\alpha$
is $\delta^{u}\tau^{u}(Q^*)$ where $u=\inf(\alpha)=-\sup(\alpha^{-1})$.
Since
$$
(a_1a_2^{2p}a_3)^*
=(a_1a_2\cdot a_2^{2p-2}\cdot a_2a_3)^*
=a_2^2 \underbrace{a_3a_4\cdots a_{2p}}_{2p-2} a_{2p+1}^2
$$
and $\tau^\ell((a_1a_2^{2p}a_3)^*)\prec Q^*$ for some $\ell\in\Z$,
$\alpha$ is of the form
$$
\alpha=\delta^u\tau^u(Q^*)=\delta^uP_1 a_i^2
(a_{i+1}a_{i+2}\cdots a_{i+2p-2})
a_{i+2p-1}^2P_2
$$
for some $i\in\Z$ and positive words $P_1$ and $P_2$.
Take a positive conjugate $\beta$ of $\alpha$ as
$$\beta=\delta^u\tau^u(a_{i+2p-1}^2P_2)P_1 a_i^2(a_{i+1}a_{i+2}\cdots a_{i+2p-2}).$$
Then $\tau^u(a_{i+2p-1}^2P_2) P_1 a_i^2
=a_{i-r+1}^{k_1}a_{i-r+2}^{k_2}\cdots a_{i}^{k_r}$,
where $i-r+1\equiv i+2p-1+u\bmod 3$
and $k_1,\ldots,k_r\ge 1$,
especially $k_1,k_r\ge 2$. (It is possible that $r=1$.)
Hence $u+r+(2p-2)\equiv 0\bmod 3$ and
$$
\beta=\delta^u a_{i-r+1}^{k_1}a_{i-r+2}^{k_2}\cdots a_{i}^{k_r}
a_{i+1}a_{i+2}\cdots a_{i+2p-2}.
$$
Since $\esyl(\beta)= u+r+(2p-2)\equiv 0\bmod 3$,
$\beta\in[\alpha]^S_0$ by Lemma~\ref{lem:xu}.
By Lemma~\ref{lem:A}, $k_q\equiv 0\bmod 2$ for some $1\le q\le r$.
Taking a positive conjugate if necessary, $\beta$ has the desired expression.
\end{proof}

\begin{proposition}\label{prop:Waa}
Let $\alpha$ be a 3-braid such that $\hat\alpha$ is a nontrivial knot.
Then $\alpha$ is conjugate to $Wa_1^{\pm 2}$ for some word $W$
such that $Wa_1^{\pm 2}$ is a shortest word in the conjugacy class of $\alpha$.
\end{proposition}

\begin{proof}
We may assume that $\alpha$ and $\alpha^{-1}$ are summit elements.
By Lemma~\ref{lem:short}, $|\alpha|$ is the shortest word length in the conjugacy class.

\case{1}{$\len(\alpha)=0$}
Let $\alpha=\delta^u$.
Taking $\alpha^{-1}$ if necessary, we may assume that $u\ge 0$,
hence $(a_2a_1)^u$ is a shortest word for $\alpha$ by Lemma~\ref{lem:short}.
If $u\in\{0,1\}$, then $\hat\alpha$ is either the 3-component unlink or the unknot.
Therefore $u\ge 2$.
Notice that
$$
(a_2a_1)^2 = (a_2a_1)(a_2a_1) = a_1a_2a_1^2.
$$
Let $W=(a_2a_1)^{u-2}a_1a_2$. Then $Wa_1^2$ is
a shortest word in the conjugacy class of $\alpha$.

\case{2}{$\len(\alpha)\ne 0$ and $\sup(\alpha)\ge 2$}
By Corollary~\ref{cor:A1}, we may assume that $\alpha=\delta^u Pa_1^2$,
where $u=\inf(\alpha)$ and $Pa_1^2$ is a nondecreasing positive word.

If $u\ge 0$, then $(a_2a_1)^u P a_1^2$ is a shortest word representing $\alpha$.
By taking $W=(a_2a_1)^u P$, we are done.
Suppose $u<0$. Since $\sup(\alpha)=u+|P|+2\ge 2$, we have $|P|\ge -u$.
Let $P=P_1P_2$, where $P_1$ is the prefix of $P$ of length $|P_1|=-u$.
By Lemma~\ref{lem:short}, $\alpha$ has a shortest word of the form
$\alpha=(P_1^*)^{-1}P_2a_1^2$.
By taking $W=(P_1^*)^{-1}P_2$, we are done.

\case{3}{$\len(\alpha)\ne 0$ and  $\inf(\alpha)\le -2$}
The braid $\alpha^{-1}$ satisfies the conditions of Case 2
because $\len(\alpha^{-1})=\len(\alpha)\ne 0$
and $\sup(\alpha^{-1})=-\inf(\alpha)\ge 2$.
Hence we are done.

\case{4}{$\len(\alpha)\ne 0$, $\sup(\alpha)\le 1$ and $\inf(\alpha)\ge -1$}
By Lemma~\ref{lem:A1}, $\len(\alpha)\ge 2$.
On the other hand, $\len(\alpha)=\sup(\alpha)-\inf(\alpha)\le 1-(-1)=2$.
Therefore
$$
\len(\alpha)=2,\quad\sup(\alpha)=1,\quad\inf(\alpha)=-1.
$$
By Corollary~\ref{cor:A1}, $\alpha$ is conjugate to $\delta^{-1} a_1^2$.
Since $\delta^{-1} a_1^2=(a_1a_3)^{-1}a_1^2=a_3^{-1}a_1$,
$\hat\alpha$ is the unknot.
This contradicts the hypothesis.
\end{proof}

\begin{remark}
Using an argument similar to the proof of Proposition~\ref{prop:Waa},
one can prove the following:
{\em Let $\alpha$ be a nonidentity 3-braid.
Then $\alpha$ is conjugate to $Wa_1^{\pm 2}$ for some word $W$ such that
$Wa_1^{\pm 2}$ is a shortest word in the conjugacy class,
unless $\alpha$ or $\alpha^{-1}$ is conjugate to one of}
$$a_1,\quad a_2a_1,\quad
a_2^{-1}a_1,\quad
(a_1a_2a_3)^k,\quad
a_2^{-1}(a_1a_2a_3)^k\qquad\text{for } k\ge 1.
$$
This property of 3-braids has been used in several papers.
For example, Ni~\cite{Ni09} and Stoimenow~\cite{Sto06}
used it in studying fibered 3-braid knots.
\end{remark}

\begin{proof}[Proof of Theorem~\ref{thm:main}]
Let $K$ be a 3-braid knot.
Because $g^*(K)\le u(K)$ holds for any knot,
it suffices to show that $u(K)\le g(K)$ holds.
We use induction on the genus $g(K)$.

If $g(K)=0$, then $K$ is the unknot, hence $u(K)=0=g(K)$.

Suppose that $g(K)\ge 1$.
By induction hypothesis, we assume that
$u(L)\le g(L)$ holds for any 3-braid knot $L$ with $g(L)< g(K)$.
By Proposition~\ref{prop:Waa}, $K$ is represented by a 3-braid $Wa_1^{\pm 2}$
for some word $W$ such that $Wa_1^{\pm 2}$ is a shortest word in its conjugacy class.
In particular,
$$g(K)=g(F_{Wa_1^{\pm 2}})=|Wa_1^{\pm 2}|/2-1=|W|/2.$$
Let $L$ be the 3-braid knot represented by $W$. Then
$$ g(L)\le g(F_W) = |W|/2-1=g(K)-1. $$
Since $L$ is obtained from $K$ by a single crossing change deleting $a_1^{\pm 2}$,
$$ u(K)\le u(L)+1. $$
By induction hypothesis, $u(L)\le g(L)$.
Therefore $u(K)\le u(L)+1\le g(L)+1\le g(K)$.
\end{proof}

The following example shows that the genus of a 3-braid knot cannot be bounded
above by a function of the unknotting number.

\begin{example}\label{eg}
For $k\ge 1$, let $\alpha_k=\delta^{-(k-1)}a_1a_2\cdots a_{k-1} a_k^2a_{k+1}^k a_{k+2}$.
By Lemma~\ref{lem:xu}, $\alpha_k\in[\alpha_k]^S_0$
because $\esyl(\alpha_k)=-(k-1)+(k+2)\equiv 0\bmod 3$.
Since $\delta^{-(k-1)}a_1a_2\cdots a_{k-1}=((a_1a_2\cdots a_{k-1})^*)^{-1}=a_{k+1}^{-(k-1)}$,
we have
$$
\alpha_k=a_{k+1}^{-(k-1)}a_k^2a_{k+1}^ka_{k+2}.
$$
From this expression, we can easily see that $\hat\alpha_k$ is a knot.
By Lemma~\ref{lem:short}, $a_{k+1}^{-(k-1)}a_k^2a_{k+1}^ka_{k+2}$ is a shortest word in
the conjugacy class, hence
$$g(\hat\alpha_k)=|\alpha_k|/2-1=k.$$
In particular, $\hat\alpha_k$ is a nontrivial knot.
If we delete $a_k^2$ from the above expression, $\alpha_k$ becomes $a_{k+1}a_{k+2}$
whose closure is the unknot.
Hence $u(\hat\alpha_k)=1$.
Therefore $\hat\alpha_k$ is a 3-braid knot with unknotting number 1 and genus $k$.
\end{example}

\section{Proof of Theorem~\ref{thm:equal}}
\label{sec:thm3}

\begin{lemma}\label{lem:pf}
Let $K$ be a 3-braid knot represented by a 3-braid
$$
a_1a_2^{2k}a_1^{-1}W \quad (k\ne 0)
$$
which is a shortest word in its conjugacy class.
Then $u(K)<g(K)$.
\end{lemma}

\begin{proof}
Taking the inverse if necessary, we may assume that $k\ge 1$.
Since $a_1a_2^{2k}a_1^{-1}W$ is a shortest word in its conjugacy class,
\begin{equation}\label{eq:gen1}
g(K)=g(F_{a_1a_2^{2k}a_1^{-1}W})
=(2k+2+|W|)/2-1=|W|/2+k.
\end{equation}
Let $L$ be the 3-braid knot represented by $W$, hence
$g(L)\le g(F_W)$.
Using Theorem~\ref{thm:main},
\begin{equation}\label{eq:gen2}
u(L)\le g(L)\le g(F_W)=|W|/2-1.
\end{equation}
Because $L$ is obtained from $K$ by $k$ crossing changes deleting $a_2^{2k}$,
\begin{equation}\label{eq:gen3}
u(K)\le u(L)+k.
\end{equation}
Combining (\ref{eq:gen1}), (\ref{eq:gen2}) and (\ref{eq:gen3}),
$$
u(K)\le u(L)+k\le |W|/2-1+k=g(K)-1<g(K).
$$

\vskip-\baselinestretch\baselineskip
\end{proof}

\begin{lemma}\label{lem:B}
Let $K$ be a 3-braid knot such that neither $K$ nor $!K$ is strongly quasipositive.
Let $K$ be represented by a 3-braid $\alpha$ such that
$\alpha$ is a summit element with Garside normal form
$$ \alpha=\delta^{-m}P. $$
\begin{enumerate}
\item
If $a_1a_2^{2k}a_3\prec P$ with $k\ge 1$ and if $m\le |P|-(2k+1)$,
then $\alpha$ is conjugate to $a_1a_2^{2k}a_1^{-1}W$
which is a shortest word in its conjugacy class.
\item
If $a_1^2 a_2\cdots a_{2k-1}  a_{2k}^2\prec P$ with $k\ge 2$ and if $m\ge 2k+1$,
then $\alpha$ is conjugate to $a_1a_2^{-2k}a_1^{-1}W$
which is a shortest word in its conjugacy class.
\item
If $a_1^2a_2^2\prec P$ and $m\ge 3$,
then $\alpha$ is conjugate to $a_1a_2^{-2}a_1^{-1}W$
which is a shortest word in its conjugacy class.
\end{enumerate}
In these cases, $u(K)<g(K)$.
\end{lemma}

\begin{proof}
If $m\le 0$ (resp.\ $m\ge |P|$), then $\alpha$ (resp.\ $\alpha^{-1}$)
is a positive braid,
hence $K$ (resp.\ $!K$) is strongly quasipositive.
Therefore $1\le m\le |P|-1$, which implies
$|\alpha|=|P|$.
Since $\alpha$ is a summit element, $|\alpha|$ is the shortest
word length in the conjugacy class of $\alpha$ by Lemma~\ref{lem:short}.

\medskip
(i)\ \
Since $a_1a_2^{2k}a_3\prec P$,
$P=P_1a_1a_2^{2k}a_3 P_2$ for some positive words $P_1$ and $P_2$, hence
$$\alpha=\delta^{-m}P_1a_1a_2^{2k}a_3 P_2.$$
By hypothesis, $m\le |P|-(2k+1)=|P_1|+|P_2|+1$,
hence $|P_1|+|P_2|\ge m-1$.
Taking a positive conjugate if necessary, we may assume $|P_1|=m-1$.
Then $\alpha$ is conjugate to
\begin{align*}
\beta
&=a_1a_2^{2k}a_3P_2\cdot \delta^{-m}P_1
=a_1a_2^{2k}a_3P_2\delta^{-1}\delta^{-m+1}P_1\\
&=a_1a_2^{2k} \cdot a_3\delta^{-1}\cdot \tau^{-1}(P_2) \cdot(P_1^*)^{-1}\\
&=a_1a_2^{2k}a_1^{-1} \cdot\tau^{-1}(P_2)\cdot(P_1^*)^{-1}.
\end{align*}
Notice that the last expression has word length
$2k+2+|P_2|+|P_1|=|P|=|\alpha|$,
so it is a shortest word in its conjugacy class.

\medskip
(ii)\ \
We will show that $\alpha^{-1}$ satisfies the condition of (i).
Because $\alpha$ is a summit element, so is $\alpha^{-1}$.
By Lemma~\ref{lem:RC}(vi), the Garside normal form of $\alpha^{-1}$ is
$\delta^{m-|P|}\tau^{m-|P|}(P^*)$.
Let $n=|P|-m$ and $Q=\tau^{m-|P|}(P^*)$. Then
$$
\alpha^{-1}=\delta^{-n}Q
$$
is the Garside normal form.
Because $|P|=|Q|$ and $m\ge 2k+1$, we have
$$ n= |P|-m \le |P|-(2k+1)=|Q|-(2k+1). $$
By Lemma~\ref{lem:RC},
$$
(a_1^2a_2\cdots a_{2k-1}a_{2k}^2)^*
=(a_1\cdot a_1a_2\cdots a_{2k-1}a_{2k}\cdot a_{2k})^*
=a_{2k+2}a_{2k+3}^{2k}a_{2k+4}.
$$
Applying $\tau$ if necessary, we may assume that
$$
a_1a_2^{2k}a_3\prec Q.
$$

So far, we have seen that $\alpha^{-1}$ satisfies the condition of (i).
Hence $\alpha^{-1}$ is conjugate to $a_1a_2^{2k}a_1^{-1}W$ which is a
shortest word in its conjugacy class.
Therefore $\alpha$ is conjugate to
$W^{-1}a_1a_2^{-2k}a_1^{-1}$
and hence to $a_1a_2^{-2k}a_1^{-1}W^{-1}$
which is also a shortest word in its conjugacy class.

\medskip
(iii)\ \
Put $k=1$ in the proof of (ii).
\end{proof}

\begin{proof}[Proof of Theorem~\ref{thm:equal}]
Let $K$ be represented by a 3-braid
$$\alpha=\delta^u a_1^{k_1}a_2^{k_2}\cdots a_r^{k_r},$$
where $u=\inf(\alpha)$ and $k_1,k_2,\ldots,k_r\ge 1$.
We may assume $\alpha\in[\alpha]^S_0$.
We will show that either $u(K)<g(K)$ holds
by Lemma~\ref{lem:B} or
$K$ is one of the knots listed in the theorem.

\medskip
If $u\ge 0$, then $\alpha$ is a positive braid.
If $r=0$, then $\alpha=\delta^u$, hence either $\alpha$ or $\alpha^{-1}$ is a positive braid.
In these cases, either $K$ or $!K$ is strongly quasipositive.
Therefore, for the proof, we may assume the following.
\begin{quote}\em
Neither $K$ nor $!K$ is a strongly quasipositive knot.\\
In particular,
$u=\inf(\alpha)\le -1$ and $r=\syl(\alpha)\ge 1$.
\end{quote}

\case{1}{$r=1$}
By Lemma~\ref{lem:A1}, $\alpha=\delta^u a_1^{2p}$ for some $p\ge 1$
and $\esyl(\alpha)=u+1\equiv 0\bmod 3$.
Since we have assumed $u\le -1$, there is an integer $m\ge 0$ such that
$$
\alpha=\delta^{-(3m+1)}a_1^{2p}.
$$

If $m=0$, then $\alpha=\delta^{-1}a_1^{2p}=(a_1a_3)^{-1}a_1^{2p}=a_3^{-1}a_1^{2p-1}$,
hence $K$ is a 2-braid knot.

If $3m+1\ge 2p$, then $\alpha^{-1}$ is a positive braid,
hence $!K$ is strongly quasipositive.

Therefore we may assume
\begin{equation}\label{eq:sylen1}
m\ge 1\quad\mbox{and}\quad 3m+1\le 2p-1.
\end{equation}
In particular, $2p\ge 3m+2\ge 5$ and $3m+1\ge 4$.
Notice that $\alpha$ is conjugate to
$$ a_2^{-2}\alpha a_2^2=
a_2^{-2}\delta^{-(3m+1)}a_1^{2p} a_2^2
=\delta^{-(3m+1)}a_1^{2p-2}a_2^2,
$$
which is also a summit element.
Since $a_1^2a_2^2\prec a_1^{2p-2}a_2^2$
and $3m+1\ge 4\ge 3$, we have $u(K)<g(K)$ by Lemma~\ref{lem:B}.

\case{2}{$r=2$}
In this case, $\alpha=\delta^u a_1^{k_1} a_2^{k_2}$.
By Lemma~\ref{lem:A}, $\esyl(\alpha)=u+2\equiv0\bmod 3$ and either $k_1$ or $k_2$ is even.
In fact, both $k_1$ and $k_2$ are even.
(Otherwise, $\alpha$ has the same induced permutation as $\delta a_1$ or $\delta a_2$
which has two cycles. This contradicts the hypothesis that $K$ is a knot.)
Since we have assumed $u\le -1$, there are integers $m\ge 0$ and $p,q\ge 1$
such that
$$
\alpha=\delta^{-(3m+2)}a_1^{2p}a_2^{2q}.
$$
In particular, $a_1^2a_2^2\prec a_1^{2p}a_2^{2q}$.

If $m\ge 1$, then $3m+2\ge 5\ge 3$, hence $u(K)<g(K)$ by Lemma~\ref{lem:B}.
Therefore we may assume that $m=0$, hence
$$\alpha=\delta^{-2}a_1^{2p}a_2^{2q}.$$

If $p=q=1$, then $\alpha=\delta^{-2}a_1^2a_2^2=a_1^{-1}a_2a_1^{-1}a_2$,
hence $K$ is the figure-eight knot.

Let $p\ge 2$.
Then $\alpha=\delta^{-2}a_1^{2p}a_2^{2q}$ is conjugate to
$$
a_3^{-1}\alpha a_3=\delta^{-2}a_1^{-1} a_1^{2p}a_2^{2q} a_3=\delta^{-2}a_1^{2p-1}a_2^{2q}a_3,
$$
which is also a summit element.
Let $P=a_1^{2p-1}a_2^{2q}a_3$.
Then $a_1a_2^{2q}a_3\prec P$ and $|P|-(2q+1)=2p-1\ge 3$,
hence $u(K)<g(K)$ by Lemma~\ref{lem:B}.
The same argument works for the case $q\ge 2$.

\case{3}{$u=-1$ or $u=-2$}
Let $m=-u$, hence $m=$ 1 or 2.
Due to Cases 1 and 2, we may assume that $r\ge 3$.
Because $\esyl(\alpha)=(-m)+r\equiv 0\bmod 3$ by Lemma~\ref{lem:A},
we have $r\equiv m\not\equiv 0\bmod 3$.
In particular, $r\ge 4$.
Let $P=a_1^{k_1}a_2^{k_2}\cdots a_r^{k_r}$, hence
$\alpha=\delta^{-m}P$.
By Lemma~\ref{lem:A}, some $k_i$ is even.
Taking a positive conjugate if necessary,
we may assume that $k_2=2p$ for some $p\ge 1$.
Since $a_1a_2^{2p}a_3\prec P$ and
$$|P|-(2p+1)=k_1+k_3+\cdots+k_r-1\ge r-2\ge 2\ge m,$$
we have $u(K)<g(K)$ by Lemma~\ref{lem:B}.

\case{4}{General Case: $r\ge 3$ and $u\le -3$}
Due to the previous cases, we may assume that
any conjugate of $\alpha$ and $\alpha^{-1}$
has syllable length $\ge 3$ and infimum $\le -3$,
and hence that
$$
\alpha=\delta^{-m}a_1^{2p}Q_1\underbrace{a_{i+1}\cdots a_{i+2q-2}}_{2q-2} Q_2\in[\alpha]^S_0
$$
for $p,q\ge 1$, $m\ge 3$ and positive words $Q_1$ and $Q_2$
with the properties in Corollary~\ref{cor:general}.
Let
$$Q=Q_1a_{i+1}\cdots a_{i+2q-2}Q_2
\quad\mbox{and}\quad
P=a_1^{2p}Q.
$$
In particular, $|P|=|Q|+2p$ and $\syl(\alpha)=\syl(P)=\syl(Q)+1$.

Recall from Corollary~\ref{cor:general} that, for each $j=1,2$,
if $|Q_j|\neq 0$ then $|Q_j|\ge 2$.

\medskip\noindent\emph{Claim.\ If $m\le |Q|-1$, then $u(K)<g(K)$.}

\begin{proof}[Proof of Claim]
Since $\syl(\alpha)\ge 3$ and $\syl(\alpha)=\syl(Q)+1$,
we have $\syl(Q)\ge 2$.
Decompose the word $Q$ into $Q=R_1R_2$,
where $R_1$ and $R_2$ are nonempty subwords of $Q$ with $\syl(Q)=\syl(R_1)+\syl(R_2)$.
Let
$$ R=\tau^{-m}(R_2)a_1^{2p}R_1\quad\mbox{and}\quad\beta=\delta^{-m}R.$$
Then $\beta$ is a positive conjugate of $\alpha$ such that $\beta\in[\alpha]^S_0$.
Hence
$$
\syl(R)=\syl(\beta)=\syl(\alpha)=\syl(P)=\syl(Q)+1=\syl(R_1)+\syl(R_2)+1.
$$
Therefore $R_1$ starts with $a_2$ and $\tau^{-m}(R_2)$ ends with $a_3$,
hence $a_3a_1^{2p}a_2\prec R$.
Since
$$m\le |Q|-1\le |P|-(2p+1) = |R|-(2p+1),$$
we have $u(K)< g(K)$ by Lemma~\ref{lem:B}.
\end{proof}

Due to the above claim, we assume $m\ge |Q|$ henceforth.

\case{4.1}{$|Q_1|=2$ or $|Q_2|=2$}
Let
$$R=\tau^{-m}(Q_2)a_1^{2p}Q_1 a_{i+1}\cdots a_{i+2q-2}
\quad\mbox{and}\quad\beta=\delta^{-m}R.$$
Then $\beta$ is a positive conjugate of $\alpha$ such that $\beta\in[\alpha]^S_0$.
Using the same argument as in the proof of the above claim, we have
$$\syl(\tau^{-m}(Q_2)a_1^{2p}Q_1)=\syl(Q_2)+1+\syl(Q_1).$$
If $|Q_1|=2$, then $Q_1=a_2^2$, hence $a_1^2 a_2^2 \prec a_1^{2p}Q_1\prec R$.
If $|Q_2|=2$, then $Q_2=a_{i+2q-1}^2$, hence
$\tau^{-m}(Q_2)=a_3^2$ and $a_3^2a_1^2 \prec \tau^{-m}(Q_2) a_1^{2p}\prec R$.
In both cases, we have $u(K)< g(K)$ by Lemma~\ref{lem:B} because $m\ge 3$.

\case{4.2}{$|Q_1|\ge 3$ or $|Q_2|\ge 3$}
Observe that
$$m\ge |Q|=|Q_1|+|Q_2|+(2q-2)\ge 3+(2q-2)=2q+1.$$

Suppose $|Q_2|\ge 3$.
Then $Q_2$ starts with $a_{i+2q-1}^2$ and
$a_1^{2p}Q_1$ ends with $a_i^2$.
(If $Q_1$ is the empty word, then $a_1^{2p}Q_1=a_1^{2p}$ and $i=1$.)
Therefore
$$a_i^2 a_{i+1} a_{i+2}\cdots a_{i+2q-2}
a_{i+2q-1}^2 \prec P.$$
Since $m\ge 2q+1$, we have $u(K)< g(K)$ by Lemma~\ref{lem:B}.

Suppose $|Q_1|\ge 3$. Let
$$
R=Q_1 a_{i+1}\cdots a_{i+2q-2} Q_2\tau^m(a_1^{2p})
\quad\mbox{and}\quad
\beta=\delta^{-m}R.
$$
Then $\beta$ is a positive conjugate of $\alpha$ such that $\beta\in[\alpha]^S_0$.
Because $|Q_1|\ge 3$, $Q_1$ ends with $a_i^2$.
Observe that $Q_2\tau^m(a_1^{2p})$ starts with $a_{i+2q-1}^2$.
(If $Q_2$ is the empty word, then $Q_2\tau^m(a_1^{2p})=\tau^m(a_1^{2p})=a_{i+2q-1}^{2p}$.)
Therefore
$$a_i^2 a_{i+1} a_{i+2}\cdots a_{i+2q-2}
a_{i+2q-1}^2 \prec R.$$
Since $m\ge 2q+1$, we have $u(K)< g(K)$ by Lemma~\ref{lem:B}.

\case{4.3}{$|Q_1|=|Q_2|=0$}
In this case, $|Q|=2q-2$ and
$$\alpha=\delta^{-m}P=\delta^{-m} a_1^{2p} a_2\cdots a_{2q-1}.$$
Since $\alpha\in[\alpha]^S_0$ and $\syl(\alpha)\ge 3$,
$\esyl(\alpha)=-m+2q-1\equiv 0\bmod 3$ by Lemma~\ref{lem:xu},
hence
$$
m\equiv 2q-1\bmod 3.
$$

We first claim that $p,q\ge 2$.
Since $\syl(\alpha)=2q-1\ge 3$, we have $q\ge 2$.
Observe that
$$P^*=(a_1^{2p} a_2\cdots a_{2q-1})^*
= a_{2q+1}^{2q-1} a_{2q+2}a_{2q+3}\cdots a_{2q+2p}.$$
By Lemma~\ref{lem:RC}(vi),
$\syl(\alpha^{-1})=\syl(P^*)=2p$.
Since $\syl(\alpha^{-1})\ge 3$, we have $p\ge 2$.

Let $R=a_1^{2p-2} a_2\cdots a_{2q-1} \tau^m(a_1^{2})$ and $\beta=\delta^{-m}R$.
Then $\beta$ is a positive conjugate of $\alpha$, hence $\beta\in[\alpha]^S$.
Since $m\equiv 2q-1\bmod 3$, $\tau^m(a_1^2)=a_{2q}^2$.
Hence
$$
\beta=\delta^{-m}R=\delta^{-m}a_1^{2p-2} a_2\cdots a_{2q-1}a_{2q}^2.
$$
Since $p\ge 2$, we have
$$a_1^{2} a_2\cdots a_{2q-1} a_{2q}^2 \prec R.$$
Therefore, if $m\ge |Q|+3$, then $m\ge|Q|+3=2q+1$, hence $u(K)< g(K)$
by Lemma~\ref{lem:B}.

Suppose that $m\le |Q|+2$.
Because we have assumed $m\ge |Q|$,
$|Q|\le m\le |Q|+2$, that is,
$$2q-2\le m\le 2q.$$
Since $m\equiv 2q-1 \bmod 3$, the above inequalities imply $m=2q-1$.
Consequently,
$$\alpha=\delta^{-2q+1}a_1^{2p} a_2\cdots a_{2q-1},$$
which is
conjugate to
$a_1^{2p} a_2\cdots a_{2q-1}\delta^{-2q+1}=a_1^{2p-1} a_2^{-2q+1}$,
hence $K$ is a connected sum of two 2-braid knots.
\end{proof}

\section*{Acknowledgments}

The first author was supported by Basic Science Research Program
through the National Research Foundation of Korea (NRF)
funded by the Ministry of Education, Science and Technology (2012R1A1A3006304).
The second author was supported by Basic Science Research Program
through the National Research Foundation of Korea (NRF) funded by
the Ministry of Education, Science and Technology (2013-014376).


\begin{thebibliography}{MM}
\bibitem[AHH12]{AHH12}
T.~Abe, R.~Hanaki and R.~Higa,
\emph{The unknotting number and band-unknotting number of a knot},
Osaka J.~Math.\ \textbf{49} (2012) 523--550.

\bibitem[Ben83]{Ben83}
D.~Bennequin,
\emph{Entrelacements et Equations de Pfaff},
Asterisque {\bf 107--108} (1983) 87--161.


\bibitem[BKL98]{BKL98}
J.S.~Birman, K.H.~Ko and S.J.~Lee,
\emph{A new approach to the word and conjugacy problems in the braid groups},
Adv.{} Math.{} \textbf{139} (1998) 322--353.


\bibitem[BW84]{BW84}
M.~Boileau and C.~Weber,
\emph{Le probl\`eme de J.~Milnor sur le nombre gordien des
noeuds alg\'ebriques},
Enseign.\ Math.\ \textbf{30} (1984) 173--222.

\bibitem[CL]{CL}
J.C.~Cha and C.~Livingston,
\emph{KnotInfo: Table of Knot Invariants},
\verb=http://www.indiana.edu/~knotinfo=.

\bibitem[KL99]{KL99}
K.H.~Ko and S.J.~Lee,
\emph{Flypes of closed 3-braids in the standard contact space},
J.~Korean Math. Soc. \textbf{36} (1999) 51--71.

\bibitem[Nak00]{Nak00}
T. Nakamura,
\emph{Four-genus and unknotting number of positive knots and links},
Osaka J.~Math. \textbf{37} (2000) 441--451.

\bibitem[Ni09]{Ni09}
Y.~Ni, \emph{Closed 3-braids are nearly fibred},
J.\ Knot Theory
Ramifications \textbf{18} (2009) 1637--1649.

\bibitem[Ras10]{Ras10}
J.~Rasmussen,
\emph{Khovanov homology and the slice genus},
Invent.\ Math.\ \textbf{182} (2010) 419--447.

\bibitem[Rud83]{Rud83}
L.~Rudolph,
\emph{Braided surfaces and Seifert ribbons for closed braids},
Comment.\ Math.\ Helv.\ \textbf{58} (1983) 1--37.

\bibitem[Rud93]{Rud93}
L.{} Rudolph,
{\em Quasipositivity as an obstruction to sliceness},
Bull.{} AMS {\bf 29} (1993) 51--59.

\bibitem[Rud99]{Rud99}
L.{} Rudolph,
\emph{Positive links are strongly quasipositive},
Geometry \&\ Topology Monographs \textbf{2} (1999)
555--562.

\bibitem[Sch85]{Sch85}
M.~Scharlemann,
\emph{Unknotting number one knots are prime},
Invent.\ Math.\ \textbf{82} (1985) 37--55.


\bibitem[ST88]{ST88}
M.~Scharlemann and A.~Thompson,
\emph{Unknotting number, genus, and companion tori},
Math.\ Ann.\ \textbf{280} (1988) 191--205.

\bibitem[Sto03]{Sto03}
A.~Stoimenow,
\emph{Positive knots, closed braids and the Jones polynomial},
Ann.\ Scuola Norm.\ Sup.\ Pisa Cl.\ Sci.\ \textbf{2}(2) (2003)
237--285.

\bibitem[Sto04]{Sto04}
A.~Stoimenow,
\emph{Polynomial values, the linking form and unknotting numbers},
Math.\ Res.\ Lett.\ \textbf{11} (2004) 755--769.


\bibitem[Sto06]{Sto06}
A.~Stoimenow,
\emph{Properties of closed 3-braids}.
arXiv:math/0606435.


\bibitem[Xu92]{Xu92}
P.~Xu,
{\em The genus of closed 3-braids},
J.{} Knot Theorey Ramifications {\bf 1} (1992) 303--326.
\end{thebibliography}
\end{document}